\documentclass[12pt]{article}

\newcommand{\be}{\begin{equation}}
\newcommand{\ee}{\end{equation}}

\usepackage{epsfig,amsmath,amssymb,graphics,graphicx} 

\usepackage{amsmath,esint}

\usepackage{amsmath,amsfonts,amsthm,amssymb,eucal}
\usepackage{breqn}

\newtheorem{Definition}{Definition}
\newtheorem{Theorem}{Theorem}
\newtheorem{Proposition}{Proposition}

\begin{document}
%%%%%%%%%%%%%%%%%%%%%%%%%%%%%%%%%%%%%%%%%%%%%%%%%%%%%%%%%%%%

\begin{center}

\textit{Computational and Applied Mathematics 43 (2024) 183}

\vskip 7mm

{\bf \large Parametric General Fractional Calculus: Nonlocal} 
\vskip 3mm
{\bf \large 
Operators Acting on Function with Respect to Another Function}
\vskip 3mm

\vskip 7mm
{\bf \large Vasily E. Tarasov}$^{1,2}$ \\
\vskip 3mm

${}^1$ {\it Skobeltsyn Institute of Nuclear Physics, \\ 
Lomonosov Moscow State University, Moscow 119991, Russia} \\
{E-mail: tarasov@theory.sinp.msu.ru} \\

${}^2$ {\it
Department of Physics, 915, \\ 
Moscow Aviation Institute (National Research University),
Moscow 125993, Russia } \\

\begin{abstract} 
In this paper, we present the definitions and some properties of the general fractional integrals (GFIs) and 
general fractional derivatives (GFDs) of a function f(x) with respect to another function g(x).
Examples of special cases of parametric GF operators are described.
Some properties of the parametric GFIs and the parametric GFDs are proved.
Representation of the parametric GFIs and parametric GFDs via the GF integrals on the finite interval (a,b).
Fundamental theorems of parametric GFC are proved.
Applied mathematical interpretations of parametric general fractional derivatives for processes with memory are given.
\end{abstract}

\end{center}



%%% 26A33: Fractional derivatives and integrals

%%%%%%%%%%%%%%%%%%%%%%%%%%%%%%%%%%%%%%%%%%%%%%%%%%%%%%%%%%%%%
%%%%%%%%%%%%%%%%%%%%%%%%%%%%%%%%%%%%%%%%%%%%%%%%%%%%%%%%%%%%%
%%%%%%%%%%%%%%%%%%%%%%%%%%%%%%%%%%%%%%%%%%%%%%%%%%%%%%%%%%%%%

\newpage

%%%%%%%%%%%%%%%%%%%%%%%%%%%%%%%%%%%%%%%%%%%%%%%%%%%%%%%%%%%%%

\section{Introduction}

Fractional calculus is a calculus of differential and integral operators of non-integer order \cite{SKM} - \cite{Handbook2019-2} 
that allow us to describe wide class of nonlocal systems and processes \cite{Handbook2019-4} - \cite{BOOK-DG-2021}.
If these operators satisfy generalizations of the fundamental theorems of standard calculus, then these operators are called fractional integrals (FIs) and fractional derivatives (FDs).

%%%\cite{SKM,Kiryakova,Podlubny,KST,Diethelm,Handbook2019-1,Handbook2019-2}
%%% \cite{Handbook2019-4,Handbook2019-5,TarasovSpringer,KLM,Mainardi,US2013,APSS2014a,APSS2014b,Povstenko,US2018,BOOK-MDPI-2020,BOOK-DG-2021}). 

To describe a wider class of systems and processes, we should consider operator with kernels belonging to a wide class of functions. 
However, this approach can lead to a mathematically inadequate description, if for such operators there are no generalizations of fundamental theorems of standard calculus.
This disadvantage can be eliminated by using general fractional calculus (GFC) \cite{Luchko2021-1} - \cite{Bazhlekova2021}. 
The GFC is based on the concepts of kernel pairs and operators, which were proposed by N.Ya. Sonin (1849-1915) in 1884 article \cite{Sonin-1,Sonin-2}. 
Note that the name "Sonin" \cite{Sonin-MathNet,Sonin-Wikipedia} is mistakenly used as "Sonine". 
The incorrect spelling of the Sonin surname in some English-language articles is due to the citation of an article in French-language journal \cite{Sonin-1}, where his surname was presented in the French-language form.
Note that the terms "general fractional calculus", "general fractional integrals" and "general fractional derivatives" were suggested by Kochubei in 2011 \cite{Kochubei2011}. 
For these operators the general fundamental theorems are proved in \cite{Kochubei2011,Hanyga,Luchko2021-1}. 

Following work \cite{Luchko2021-1}, one can conditionally distinguish three types of GFC, which are based on 
three different types of pairs of spaces (classes, sets) to which pairs of Sonin kernels belong: 
({\cal K}) the Kochubei's GFC proposed in 2011, \cite{Kochubei2011} (see also \cite{Kochubei2019-1,Kochubei2019-2});
({\cal H}) the Hanyga's GFC proposed in 2020, \cite{Hanyga};
({\cal L}) the Luchko's GFC proposed in 2021, \cite{Luchko2021-1}.

The GFC allows us to describe a wider type of non-localities in space and time. 
The most convenient form of GFC is the Luchko GFC \cite{Luchko2021-1} - \cite{Luchko2023}. 
The Luchko GFC is also developed and applied to various sciences in \cite{Mathematics2021-3} - \cite{Atanackovic2023-2}. 
Other forms of GFC is described and applied in \cite{Kochubei2011} - \cite{Bazhlekova2021}.
In 2023 paper \cite{Luchko2023}, the GFC on finite interval is proposed, 
The general fractional integrals (GFIs) and general fractional derivatives (GFDs) form the GFC, in which generalizations of the fundamental theorems of standard calculus are satisfied. 
The following basic properties of GFDs and GFIs on the finite interval $(a,b)$ that are proved in \cite{Luchko2023}: 
The GFIs satisfy the semi-group and commutativity properties;
Equations expressing GFDs of the Riemann-Liouville type in terms of GFDs of the Caputo type;
The fundamental theorems of GFC;
the GFDs satisfy rules of integration by parts;
the GFIs satisfy rules of integration by parts.

In this paper, we present the definitions and some properties of the GFIs and GFDs of a function $f(x)$ with respect to another function $g(x)$.
In standard calculus of integrals and derivatives of integer orders, the integrals and derivatives of a function $f(x)$ with respect to another function $g(x)$ is usually called the parametric integrals and derivatives.
Therefore the proposed generalization of the Luchko GFC can be also called as "Parametric GFC" for convenience instead of a long phrase "GFC of GF integrals and GF derivatives of a function with respect to another function". 
In this paper, to simplify the terminology, we proposed to use "parametric GF operator" instead of "general fractional operator of a function f(x) with respect to another function g(x)". This not only greatly simplifies the terminology, but also provides a connection with terms actively used in the standard calculus of derivatives and integrals of integer order.

%%%%%%%%%%%%%%%%%%%%%%%%%%%%%%%%%%%%%%%%%%%%%%%%%%%%%%%%%%%%%
%%%%%%%%%%%%%%%%%%%%%%%%%%%%%%%%%%%%%%%%%%%%%%%%%%%%%%%%%%%%%
%%%%%%%%%%%%%%%%%%%%%%%%%%%%%%%%%%%%%%%%%%%%%%%%%%%%%%%%%%%%%

In standard calculus, a parametric derivative can be considered as a derivative of a dependent variable $f(x)$ with respect to another dependent variable $g(x)$ that is taken when both variables depend 
on an independent third variable $x$.
The parametric derivative of integer order of a function $f(x)$ with respect to another function $g(x)$ is
\[
D^n_g \, f(x) \, = \, 
\left(\frac{1}{g^{(1)}(x)} \, \frac{d}{dx}\right)^n \, f(x) ,
\]
where $n \, \in \, \mathbb{N}$ and 
$g^{(1)}(x) \, = \, dg(x)/dx$ is the first-order derivative.
A derivation of generalized Cauchy formula for the case of parametric integration of integer order was proposed by Shelkovnikov \cite{Shelkovnikov1951} in 1951. 

In standard calculus of integrals and derivatives of integer orders, the parametric integral and derivatives are directly connected with the standard chain rule (for example, see Sec. 5.2.5 of \cite{Zorich}, p.204, and Sec. 5.7 of \cite{ISS1979}, p.238, Sec. 14.2 of \cite{ISS1979}, pp.669-670).
In the fractional calculus of integrals and derivatives of non-integer orders the generalization chain rules are violated 
\cite{CNSNS2016,Cresson2020,Math2019}.
The generalized equation of the chain rule for fractional derivatives (see equation 2.209 in Sec 2.7.3 of \cite{Podlubny}, pp.97-98) becomes very complicated so that it cannot be used in applications.
Therefore the importance of the parametric GFIs and GFDs is greatly increased.

%%%%%%%%%%%%%%%%%%%%%%%%%%%%%%%%%%%%%%%%%%%%%%%%%%%%%%%%%%%%%
%%%%%%%%%%%%%%%%%%%%%%%%%%%%%%%%%%%%%%%%%%%%%%%%%%%%%%%%%%%%%
%%%%%%%%%%%%%%%%%%%%%%%%%%%%%%%%%%%%%%%%%%%%%%%%%%%%%%%%%%%%%

Let us briefly describe some of the main works on parametric fractional calculus. To do this, we will conditionally divide these works by century (XIX, XX and XXI).

XIX) The idea about fractional integral of a function by another function in fact was suggested by Liouville in 1835 work \cite{Liouville1935}. 
The fractional integral of a function by another function was in fact proposed in 1865 by Holmgren \cite{Holmgren1865}, p.10. 

XX) The papers, where the notion of differentiation of a function by another one appeared again, are those of 
Erdelyi \cite{Erdelyi1964}, 
Talenti \cite{Talenti1965} and also 
Erdelyi \cite{Erdelyi1970}. 
Some simple properties of fractional integrals (FIs) of a function by another were considered by Chrysovergis \cite{Chrysovergis1971} in 1971. 
Note that such integrals in an implicit form are proposed by Sewell in Chapter 3 of 1937 work \cite{Sewell1937}, p.14, in the proof of the invariance property of fractional differentiability of function given on curves under conformal mappings. 
The parametric FIs and FDs of a function by another in the complex plane were investigated by Osler in \cite{Osler1970a,Osler1970b,Osler1972a, Osler1972b} in 1970 and 1972. 
In papers \cite{Osler1970a,Osler1972a}, generalization of the Leibniz role for parametric FDs are proposed.
In work \cite{Osler1970a,Osler1972a}, an integral form of the generalized Leibniz role for parametric FDs is considered.
In paper \cite{Osler1970b}, a generalization of the chain rule for the parametric FDs is proposed.
Fractional differentiation of a function by another in the Grunwald-Letnikov form was studied by Krasnov \cite{Krasnov1977}. 
Some historical comments on can see in Sec. 23.1 of \cite{SKM}, pp.431-432.

%%%%%%%%%%%%%%%%%%%%%%%%%%%%%%%%%%%%%%%%%%%%%%%%%%%%%%%%%%%%%

In usual fractional calculus, the parametric fractional integrals and derivatives, are described in classical books \cite{SKM,KST}.
The parametric Riemann-Liouville fractional derivative is described in Section 18.2 of book \cite{SKM}, pp.325-329, and Section 2.5 in \cite{KST}, pp.99-105.

%%%%%%%%%%%%%%%%%%%%%%%%%%%%%%%%%%%%%%%%%%%%%%%%%%%%%%%%%%%%%

XXI) It should be noted that in the first decade of the 21st century there was practically no interest with the parametric fractional derivative and integrals. 
In fact, the active development of calculus for parametric fractional operators sharply increased from the second half of the second decade of the 21st century.

XXI-1) Note that in 2012, an application of parametric FDs to change of variables and nonlocal symmetries of equations with FDs is considered by Gazizov, Kasatkin and Lukashchuk in \cite{Gazizov2012}.

%%%In the general case, an arbitrary change of variables does not preserve 
%%%the form of the fractional differential operator.
%%% In particular, this substitution transforms the Riemann-Liouville FDs 
%%% to the left-sided FDs of the function with respect to the function.

%%%%%%%%%%%%%%%%%%%%%%%%%%%%%%%%%%%%%%%%%%%%%%%%%%%%%%%%%%%%%

%%% Almeida and Malinowska 

XXI-2) Parametric Caputo fractional derivative 
was first proposed in equation 23 of Definition 3 in paper
\cite{Elasticity1}, p.224.
This work also describes the use of the parametric Caputo FDs
to generalize the elasticity of economic processes with memory. 
Then, the properties of this derivative were described in \cite{Almeida-1} in 2017.
In 2017-2021, results about the $\psi$-Caputo fractional derivatives and integrals are derived in papers 
\cite{Almeida-1,Almeida-2,Almeida-3,Almeida-4,Almeida-5,Almeida-6}. 
In 2018, Almeida, Malinowska and Monteiro \cite{Almeida-2} considered equations with the parametric Caputo FDs and their applications. 
In 2019, Almeida in \cite{Almeida-3} described some additional properties of the Osler's parametric FIs and FDs derivatives.
In the same year, Almeida, Jleli and Samet \cite{Almeida-4} gave numerical study of equations with $\psi$-Caputo FDs, which describe the fractional relaxation-oscillation.
In 2020, Almeida \cite{Almeida-5} considered equations involving the $\psi$-Caputo FDs.
Then, Almeida, Malinowska and Odzijewicz \cite{Almeida-6} studied systems of equations with the parametric Caputo FDs in 2021.

%%%%%%%%%%%%%%%%%%%%%%%%%%%%%%%%%%%%%%%%%%%%%%%%%%%%%%%%%%%%%

%%% Sousa and Oliveira 

XXI-3) In 2018-2020, the results about the parametric generalization of the Hilfer FDs and FIs are proposed by 
\cite{Sousa-1,Sousa-2,Sousa-3,Sousa-4,Sousa-5,Sousa-6,Sousa-7,Sousa-8}, where these operators are called $\psi$-Hilfer FIs and FDs.
In 2018, Sousa and Oliveira \cite{Sousa-1} proposed the $\psi$-Hilfer fractional derivative. 
Then they considered \cite{Sousa-2} the Ulam-Hyers-Rassias stability for nonlinear equations with the $\psi$-Hilfer fractional operators. 
In 2019, Sousa and Oliveira \cite{Sousa-3} proposed the $\psi$-fractional integrals. 
In the same year, Sousa and Oliveira \cite{Sousa-4,Sousa-5} described 
the Gronwall inequality, the Cauchy type problem and Leibniz type rule for the $\psi$-Hilfer fractional operators. 
The nonlinear equations with $\psi$-Hilfer fractional derivatives are considered in Kucche, Mali and Sousa \cite{Sousa-6} in 2019.
In 2020, Sousa, Gastao and Oliveira proposed so-called $\psi$-Hilfer pseudo-fractional operators.
In the same year, Sousa, Machado and Oliveira \cite{Sousa-8} proposed the $\psi$-Hilfer FC of variable order. 

%%%%%%%%%%%%%%%%%%%%%%%%%%%%%%%%%%%%%%%%%%%%%%%%%%%%%%%%%%%%%

%%% Fernandez, and Fahad 

XXI-4) In 2021, Oumarou, Fahad, Djida and Fernandez \cite{Fernandez-1} proposed the parametric FC with analytic kernels with respect to functions.
In the same year, Fahad, Fernandez, Rehman and Siddiqi \cite{Fernandez-2} described parametric generalization of tempered and Hadamard-type FC.
In 2021 work \cite{Fernandez-3}, Fernandez, Restrepo and Djida proposed the parametric fractional Laplacian of a function with respect to another function.
Then in 2021, Fahad and Fernandez \cite{Fernandez-4} developed
operational calculus for equations with the parametric Caputo FDs in 2021.
In 2022, Mali, Kucche, Fernandez, and Fahad \cite{Fernandez-5}
proposed parametric tempered FC and consider equations with 
the parametric tempered FDs.
In the same year, Kucche, Mali, Fernandez and Fahad \cite{Fernandez-6} consider the parametric tempered Hilfer FDs and the equations with such FDs.
In 2019-2023, Fahad, Rehman and Fernandez \cite{Fernandez-7} 
the Laplace transforms of 
the parametric fractional differential and integral operators of non-integer orders
and their applications to equations with parametric FDs.
Note that a formulation of operational calculus for 
the parametric differential and integral operators of integer orders has been proposed by Rapoport \cite{Rapoport} in 1970 
(see also \cite{Brychkov-1,Brychkov-2}). 
The operational calculus for the general fractional derivatives with the Sonin kernels has been suggested 
by Luchko \cite{Luchko2021-3} in 2021. 
Then, in 2022, Al-Kandari, Hanna, and Luchko \cite{Luchko2022-4} give the operational calculus 
for the general fractional derivatives of arbitrary order.
Then, in 2022, Al-Kandari, Hanna, and Luchko \cite{Luchko2022-4} give the operational calculus for the general fractional derivatives of arbitrary order. 
In 2023, Al-Refai and Fernandez 
\cite{Fernandez-2023} consider a generalization of the fractional calculus with Sonin kernels via conjugations in the form $D^{*} \, = \, S \, D \, S^{-1}$ with substitution operator $S$. 
In the same year, Fernandez \cite{Fernandez-2023-NEW} proposed
operational calculus for general conjugated fractional derivatives. 
It should be noted that the so-called "conjugated" fractional operators can be considered as more general than the parametric fractional operators. 
The operator $S$ can be considered as is a general invertible linear operator, whereas substitution operator $S=Q_g$ (see \cite{SKM}, p.326, and \cite{KST}, p.100) is the specific operator of composition with a monotonic function $g$. Therefore, the theory of "conjugated" operators also includes the left-sided and right-sided FDs and FIs, the weighted FDs and FIs, and some others in addition to the case parametric FDs and FIs.

%%%%%%%%%%%%%%%%%%%%%%%%%%%%%%%%%%%%%%%%%%%%%%%%%%%%%%%%%%%%%
%%%%%%%%%%%%%%%%%%%%%%%%%%%%%%%%%%%%%%%%%%%%%%%%%%%%%%%%%%%%%
%%%%%%%%%%%%%%%%%%%%%%%%%%%%%%%%%%%%%%%%%%%%%%%%%%%%%%%%%%%%%

XXI-5) In 2023 paper \cite{FF-2023}, a GFC of operators that is defined through the Mellin convolution instead of the Laplace convolutional operators of the Luchko GFC is proposed. 
The operators are generalizations of standard scaling operator for the case of general form of nonlocality. 
The usual Hadamard and Hadamard-type FIs and FDs are special case of the proposed Mellin convolutional GF operators. 
The Mellin convolutional GFIs and GFDs can be considered as parametric FIs and FDs of function $f(x)$ with respect to function $g(x) \, = \, \ln(x/a)$ with $x>a>0$.

In this paper, we proposed definitions of the parametric GFIs and GFDs of a function $f(x)$ with respect to another function $g(x)$. 
Some properties of the parametric GFIs and GFDs are proved.
The proposed results can be considered as a generalization of the Luchko results that are derived in \cite{Luchko2023}
The proposed parametric GFC can also be considered as generalization of the Mellin convolutional GFIs and GFDs 
\cite{FF-2023}, from the function $g(x) \, = \, \ln(x/a)$ to wider class of functions $g(x)$.
The first and second fundamental theorems of the parametric GFIs and GFDs are proved.
These proofs are based on the properties of the substitution operators and properties of GFIs and GFDs on the finite intervals $(a,b)$.
Therefore, the proved fundamental theorems allow us to state that the parametric GFIs and GFDs form a parametric GFC.

Let us note main differences of the proposed work and paper \cite{Fernandez-2023} is the consideration of a parametric GFC on finite intervals instead of GFC on infinite positive semiaxis. 
The suggested work can be considered as extension of the Al-Refai-Luchko paper \cite{Luchko2023}about general fractional calculus in finite intervals. 
The work \cite{Fernandez-2023} can be considered as a parametric extension of general fractional calculus proposed in the Luchko paper \cite{Luchko2021-1}. 
One can say that difference of the proposed manuscript and the Al-Refai-Fernandez paper \cite{Fernandez-2023} is similar to difference of between Al-Refai-Luchko paper \cite{Luchko2023} and Luchko paper \cite{Luchko2021-1} in the part of parametric general fractional operators.
Let us note the difference in terminology, which lies in the fact that parametric fractional operators are called the "conjugate" operators in the work \cite{Fernandez-2023}, which does not coincide with the terminology of the standard calculus of integrals and derivatives of integer order.

In Section 2, set of functions and set of kernel pairs are defined.
In Section 3, parametric GF integrals and parametric GF derivatives of $f(x)$ with respect to $g(x)$ are defined.
Examples of special cases of parametric GF operators are described.
Some properties of the parametric GFIs and the parametric GFDs are proved.
Representation of the parametric GFIs and parametric GFDs via the GF integrals on the finite interval $(a,b)$.
In Section 4, fundamental theorems of parametric GFC are proved.
In Section 5, economic interpretations of parametric GF derivatives are point out.

%%%%%%%%%%%%%%%%%%%%%%%%%%%%%%%%%%%%%%%%%%%%%%%%%%%%%%%%%%%%%
%%%%%%%%%%%%%%%%%%%%%%%%%%%%%%%%%%%%%%%%%%%%%%%%%%%%%%%%%%%%%
%%%%%%%%%%%%%%%%%%%%%%%%%%%%%%%%%%%%%%%%%%%%%%%%%%%%%%%%%%%%%

%%% \newpage

%%%%%%%%%%%%%%%%%%%%%%%%%%%%%%%%%%%%%%%%%%%%%%%%%%%%%%%%%%%%%

\section{Set of Functions and Set of Kernels}

%%%%%%%%%%%%%%%%%%%%%%%%%%%%%%%%%%%%%%%%%%%%%%%%%%%%%%%%%%%%%

\subsection{Set of functions of GFC}

In this paper, definitions and some propertied of parametric GFIs and GFDs are considered on the finite interval $(a,b)$ with $ - \, \infty \, < \, a \, < \, b \, < \, +\infty$ are considered. 

In this section, we define the set of functions ${\cal G}(\Omega)$, $C_{-1,g}(\Omega)$, $C^1_{-1,g}(\Omega)$, where $\Omega \, \subset \, \mathbb{R}$ is finite interval of the form $(a,b]$ or $[a,b)$. 

1) Let $g(x)$ be an increasing and positive monotone function on $(a, b)$, having a continuous first-order derivative 
on $(a,b]$ (or on $[a,b)$) such that $g^{(1)}(x) \, \ne \, 0$ for all $x \, \in \, [a,b]$.
Then the set of such functions will be denoted as ${\cal G}(a,b]$ (or as ${\cal G}[a,b)$).

2) Function $f(x)$ belongs to the set $C_{-1,g}(a,b]$, if it
can be represented as
\[
f(x) \, = \, (g(x) \, - \, g(a))^{p} \, m_p(x) ,
\]
where $m_p(x) \, \in \, C[a,b]$, and $p \, > \, - \, 1$.

Function $f(x)$ belongs to the set $C_{-1,g}(a,b)$, if it
can be represented as
\[
f(x) \, = \, (g(b) \, - \, g(x))^{q} \, m_q(x) ,
\]
where $m_q(x) \, \in \, C[a,b]$, and $q \, > \, - \, 1$.

3) The condition $f(x) \, \in \, C^1_{-1,g}(a,b]$ means that first-order derivatives of function $f(x)$ can be represented as
\[
\frac{d f(x)}{d x} 
\, = \, (g(x) \, - \, g(a))^{p} \, v_p(x) ,
\]
where $v_p(x) \, \in \, C[a,b]$, $p \, > \, - \, 1$. 

The condition $f(x) \, \in \, C^1_{-1,g}[a,b)$
means that first-order derivatives of function $f(x)$ can be represented as
\[
\frac{d f(x)}{d t} 
\, = \, (g(b) \, - \, g(x))^{q} \, v_q(x) ,
\]
where $v_q (x) \, \in \, C[a,b]$, $q \, > \, - \, 1$.

%%%%%%%%%%%%%%%%%%%%%%%%%%%%%%%%%%%%%%%%%%%%%%%%%%%%%%%%%%%%%
\subsection{Sonin and Luchko sets of kernel pairs}

%%%In \cite{Luchko-2023}, 

Let $(a,b)$ with 
$ - \, \infty \, < \, a \, < \, b \, < \, +\infty$ be a finite interval of the real line $\mathbb{R}$. 

The GFC is formulated for the Sonin set ${\cal S}_{1}(\mathbb{R}_{+})$ of kernel pairs $( M(x) , \, K(x))$.

\begin{Definition}
Let a kernel pair $( M(x) , \, K(x))$ with 
$x \, \in \, \mathbb{R}_{+} = (0,\infty)$ 
satisfy the Sonin condition
\begin{equation} \label{S-C}
\int^{x}_{0} M (x \, - \, z) \, 
 K (z) \, dz \, = \, \{1\} \, = \, 
\begin{cases}
1 & \mbox{if } x \in (0,b-a] \\
0 & \mbox{if } x \notin (0,b-a] ,
\end{cases}
\end{equation}
Then, such set of kernel pairs is called the Sonin set of kernel pairs.
\end{Definition}

The kernels $M(x)$ and $K(x)$ as functions can be considered to belong to different spaces or classes of functions. 
The types of general fractional calculus can be conditionally divided depending on additional conditions imposed on these functions \cite{Luchko2021-1}. 
Among the main types of calculus, the following can be distinguished: 
({\cal K}) the Kochubei's GFC \cite{Kochubei2011} (see also \cite{Kochubei2019-1,Kochubei2019-2});
({\cal H}) the Hanyga's GFC \cite{Hanyga};
({\cal L}) the Luchko's GFC \cite{Luchko2021-1}.

Let us give a definition of the definition of the Luchko set ${\cal L}_{1}(\mathbb{R}_{+})$ of kernel pairs. 

\begin{Definition}
Let a kernel pair $( M(x) , \, K(x))$ with 
$x \, \in \, \mathbb{R}_{+} = (0,\infty)$ 
satisfy the Sonin condition \eqref{S-C} and
\begin{equation}
M(x) , \, K(x) \, 
\in \, C_{-1} (0,b-a] ,
\end{equation}
where $F(x) \, \in \, C_{-1} (0,b-a]$ 
if there is such a function 
$G(x) \, \in \, C[0,b-a]$ 
that $F(x) \, = \, (x \, - \, a)^{p} \, G(x)$ 
with $p \, > \, -1$. 

Then, such set of kernel pairs is called the Luchko set ${\cal L}_{1}(\mathbb{R}_{+})$. 
\end{Definition}

As examples of kernel pairs $(M(x), \, K(x))$, which 
belong to the Luchko set, one can give the following. 

First example of the kernel pair
%%% 1
\[
M(x) \, = \,
h_{\alpha}(\lambda x) \, = \, 
\frac{ (\lambda \, x)^{\alpha -1}}{\Gamma(\alpha)} ,
\]
\begin{equation} \label{MK-Example-1} 
K(x) \, = \,
\lambda \, h_{1-\alpha}(\lambda x) \, = \, 
\frac{\lambda \, (\lambda \, x)^{-\alpha}}{\Gamma(1-\alpha)} .
\end{equation}

Second example of the kernel pair
%%% 2
\[
M(x) \, = \,
h_{\alpha,\lambda}(\lambda \, x) \, = \, 
\frac{(\lambda \, x)^{\alpha-1}}{\Gamma(\alpha)} e^{- \lambda \, t} ,
\]
\begin{equation} \label{MK-Example-2}
K(x) \, = \, 
\lambda \, h_{1-\alpha,\lambda}(\lambda x) \, + \, 
\frac{\lambda}{\Gamma(1-\alpha)} \, 
\gamma(1-\alpha,\lambda x) .
\end{equation}

Third example of the kernel pair
%%% 3
\[
M(x) \, = \, (\lambda \, x)^{\beta -1} \, 
E_{\alpha,\beta} [-(\lambda \, x)^{\alpha} ] ,
\]
\begin{equation} \label{MK-Example-3}
K(x) \, = \,
\frac{\lambda \, (\lambda \, x)^{\alpha - \beta}}{\Gamma(\alpha-\beta+1)} 
\, + \,
\frac{\lambda \, (\lambda \, x)^{- \beta}}{\Gamma(1-\beta+1)} .
\end{equation}
In equations \eqref{MK-Example-1}, \eqref{MK-Example-2}, \eqref{MK-Example-3}, $\gamma(\beta,x)$ is the incomplete gamma function, $E_{\alpha,\beta}[x]$ is the two-parameters Mittag-Leffler function, where $\lambda >0$, $[\lambda]=[x]^{-1}$, 
$0 \, < \, \alpha \, \le \, \beta \, < \, 1$, and $x \, > \, 0$. 

Note that the parametric is used to obtain the standard physical dimension of the described quantities when using GFC.
The physical dimensions of the kernels are $[M(x)]=[1]$ and $[K(x)]=[x]^{-1}$.

For other examples see Table 1 of \cite{GNCM2022}, pp.5-7, Table 1 of \cite{AP2022}, p.15, \cite{PA2023}, p.11,
\cite{Math-Multi2023}, pp.21-22, \cite{Entropy2023-Prob}, p.10). 
These examples can be extended by the kernel pairs 
$(M_{new}(x)= \lambda^{-1} K(x), K_{new}(x)=\lambda M(x))$, 
where $(M(x), K(x))$ are pairs of this list of examples.

%%%%%%%%%%%%%%%%%%%%%%%%%%%%%%%%%%%%%%%%%%%%%%%%%%%%%%%%%%%%%

\subsection{Substitution operator}

Let us define the substitution operator $Q_g$
(see \cite{SKM}, p.326, and equation 2.5.10 in \cite{KST}, p.100). 

\begin{Definition}
Let $g(x)$ be an increasing and positive monotone function on $[a, b)$ 
with $ - \, \infty \, < \, a \, < \, b \, < \, +\infty$, 
having a continuous first-order derivative $g^{(1)}(x) \, = \, dg(x)/dx$ for all $(a,b)$.

If $g^{(1)}(x) \, \ne \, 0$ for all $x \, \in \, (a,b)$,
then the substitution operator $Q_g$ is 
\begin{equation}
(Q_g \, f)(x) \, = \, f(g(x)) ,
\end{equation}
and $Q^{-1}_{g}$ is its inverse operator
\begin{equation}
(Q^{-1}_g \, f)(x) \, = \, f(g^{-1}(x)) 
\end{equation}
such that
\begin{equation}
( Q^{-1}_g \, Q_g \, f)(x) \, = \,
(Q_g \,Q^{-1}_g \, f)(x) \, = \, f(x) .
\end{equation}
\end{Definition}

Note that the condition ensuring the existence of an inverse function for the function 
$y \, = \, g(x)$ has the form \cite{ISS1979}, pp.671-672, states:
If the function $y \, = \, g(x)$ has a non-zero and sign-preserving derivative in some neighborhood of the point
$x_0$, then for this function of the neighborhood of the point $x_0$ there exists an inverse function $x \, = \, g^{-1}(y)$,
defined and differentiable in some neighborhood of the point $y_0$, where $y_0 \, = \, g(x_0)$.
The derivative of this inverse function at the point $y_0$ is equal to $(g^{(1)}(x_0))^{-1} \, = \, 1/ g^{(1)}(x_0)$.

For example one can consider the operator $Q_g$ 
with $g(x) \, = \, x \, - \, a$ such that
\[
Q_g \, f(x) \, = \, f(x \, - \, a) ,
\]
which is the shift operator that is usually denoted as $\tau_a$ 
(see equation 5.10 in \cite{SKM} and equation 1.3.27 in \cite{KST}), and the inverse operator is
\[
Q^{-1}_g \, f(x) \, = \, f(x \, + \, a) .
\]

Let us prove the property of the substitution operator.

\begin{Proposition} \label{Property-Q}
Let $g(x)$ be an increasing and positive monotone function on $(a, b]$, having a continuous first-order derivative $g^{(1)}(x)$ on $(a, b)$, and let $F(x) \, \in \, C^1(a,b]$.
Then,
\begin{equation} \label{Property-DQ-QD}
\left( \frac{1}{g^{(1)}(u)} \frac{d}{du} \right) \, Q_{g} \, F(u) \, = \, 
Q_g \, F^{(1)}(u) .
\end{equation}
\end{Proposition} 

\begin{proof}
Using that
\[
dg(u) \, = \, g^{(1)}(u) \, du ,
\]
\[
Q_g \, f(u) \, = \, f(g(u)) ,
\]
one can get
\[
\left( \frac{1}{g^{(1)}(u)} \frac{d}{du} \right) \, Q_{g} \, F(u) \, = \, 
\left( \frac{d}{dg(u)} \right) \, F(g(u)) \, = \, 
\]
\[
\left( \frac{dF(g(u))}{dg(u)} \right) \, = \, 
\left( \frac{dF(z)}{dz} \right)_{z=g(u)} \, = \, 
Q_g \, F^{(1)}(u) ,
\]
which was to be proved. Q.E.D.
\end{proof}

%%%%%%%%%%%%%%%%%%%%%%%%%%%%%%%%%%%%%%%%%%%%%%%%%%%%%%%%%%%%%
%%%%%%%%%%%%%%%%%%%%%%%%%%%%%%%%%%%%%%%%%%%%%%%%%%%%%%%%%%%%%
%%%%%%%%%%%%%%%%%%%%%%%%%%%%%%%%%%%%%%%%%%%%%%%%%%%%%%%%%%%%%

%%% \newpage

\section{Parametric GF Operators of $f(x)$ with Respect to $g(x)$}

In this section, we present the definitions and some properties of the parametric general fractional integrals and 
general fractional derivatives of a function $f(x)$ with respect to another function $g(x)$.

%%%%%%%%%%%%%%%%%%%%%%%%%%%%%%%%%%%%%%%%%%%%%%%%%%%%%%%%%%%%%

\subsection{Definition parametric GF operators}

Let us give definitions of the parametric GFIs of a function $f(x)$ with respect to another function $g(x)$.

\begin{Definition}
Let $g(x)$ be an increasing and positive monotone function on $(a, b]$, having a continuous first-order derivative $g^{(1)}(x)$ on $(a, b)$, and let $( M(x) , \, K(x)) \, \in \, {\cal L}_{1}(\mathbb{R}_{+})$. 
 
Then, the left-sided parametric GFI of a function $f(x)$ with respect to function $g(x)$ on $(a,b)$ is defined by
\begin{equation}
(I^{(M)}_{a+,g} f)(x) \, = \, 
\int^x_a M(g(x)-g(u)) \, f(u) \, g^{(1)}(u) \, du ,
\end{equation}
if $f(x) \, \in \, C_{-1,g}(a,b]$, 
where $a \, < \, x \, \le \, b$.
\end{Definition}

\begin{Definition}
Let $g(x)$ be an increasing and positive monotone function on $[a, b)$, 
having a continuous first-order derivative $g^{(1)}(x)$ on $(a, b)$, and let $( M(x) , \, K(x)) \, \in \, {\cal L}_{1}(\mathbb{R}_{+})$. 

Then, the right-sided parametric GFI of a function $f(x)$ with respect to function $g(x)$ on $(a,b)$ is defined by
\begin{equation}
(I^{(M)}_{b-,g} f)(x) \, = \, 
\int^b_x M(g(u)-g(x)) \, f(u) \, g^{(1)}(u) \, du ,
\end{equation}
if $f(x) \, \in \, C_{-1,g}(a,b]$,
where $a \, \le \, x \, < \, b$.
\end{Definition}

%%%%%%%%%%%%%%%%%%%%%%

Let us give definitions of the parametric GFDs of a function $f(x)$ with respect to another function $g(x)$, where
$g^{(1)}(x) \, \ne \, 0$ for $x \, \in (a,b)$.

\begin{Definition}
Let $g(x)$ be an increasing and positive monotone function on $(a, b]$, 
having a continuous first-order derivative $g^{(1)}(x) \, \ne \, 0$ on $(a, b)$, and let $( M(x) , \, K(x)) \, \in \, {\cal L}_{1}(\mathbb{R}_{+})$. 

Then, the parametric GFDs of the Riemann-Liouville type of function $f(x)$ with respect to function $g(x)$ are defined as
\begin{equation} \label{PGFD-RL-l}
(D^{(K)}_{a+,g} f)(x) \, = \, 
\frac{1}{g^{(1)}(x)} \, \frac{d}{dx}
\int^x_a K(g(x)-g(u)) \, f(u) \, g^{(1)}(u) \, du ,
\end{equation}
if $f(x) \, \in \, C^1_{-1,g}(a,b]$, 
where $a \, < \, x \, \le \, b$, and
\begin{equation} \label{PGFD-RL-r}
(D^{(K)}_{b-,g} f)(x) \, = \, 
- \, \frac{1}{g^{(1)}(x)} \, \frac{d}{dx}
\int^b_x K(g(u)-g(x)) \, f(u) \, g^{(1)}(u) \, du ,
\end{equation}
if $f(x) \, \in \, C^1_{-1,g}[a,b)$,
where $a \, \le \, x \, < \, b$.

Operator \eqref{PGFD-RL-l} is called left-sided parametric GFD of the RL type, and operator \eqref{PGFD-RL-r} is called right-sided parametric GFD of the RL type.

\end{Definition}

The parametric GFDs of the Riemann-Liouville type can be expressed via the parametric GFIs by the equations
\begin{equation} 
(D^{(K)}_{a+,g} f)(x) \, = \, 
\frac{1}{g^{(1)}(x)} \, \frac{d}{dx} \, 
(I^{(K)}_{a+,g} f)(x) ,
\end{equation}
where $a \, < \, x \, \le \, b$, and
\begin{equation}
(D^{(K)}_{b-,g} f)(x) \, = \, 
- \, \frac{1}{g^{(1)}(x)} \, \frac{d}{dx} \, 
(I^{(K)}_{b-,g} f)(x) ,
\end{equation}
where $a \, \le \, x \, < \, b$.

\begin{Definition}
Let $g(x)$ be an increasing and positive monotone function on $(a, b]$, 
having a continuous first-order derivative $g^{(1)}(x) \, \ne \, 0$ on $(a, b)$, and let $( M(x) , \, K(x)) \, \in \, {\cal L}_{1}(\mathbb{R}_{+})$. 

Then, the parametric GFDs of the Caputo type of function $f(x)$ with respect to function $g(x)$ are defined as
\begin{equation} \label{PGFD-C-l}
(D^{(K),*}_{a+} f)(x) \, = \, 
\int^x_a K(g(x)-g(u)) \, f^{(1)}(u) \, du ,
\end{equation}
if $f(x) \, \in \, C^1_{-1,g}(a,b]$,
where $a \, < \, x \, \le \, b$, and
\begin{equation} \label{PGFD-C-r}
(D^{(K),*}_{b-,g} f)(x) \, = \, 
- \, \int^b_x K(g(u)-g(x)) \, f^{(1)}(u) \, du ,
\end{equation}
if $f(x) \, \in \, C^1_{-1,g}[a,b)$,
where $a \, \le \, x \, < \, b$.

Operator \eqref{PGFD-C-l} is called left-sided parametric GFD of the Caputo type, and operator \eqref{PGFD-C-r} is called right-sided parametric GFD of the Caputo type.
\end{Definition}

The parametric GFDs of the Caputo type can be expressed via the parametric GFIs by the equations
\[
(D^{(K),*}_{a+,g} f)(x) \, = \, 
\left(I^{(K)}_{a+,g} \frac{1}{g^{(1)}(u)} \, \frac{d}{du} \, f(u) \right)(x) ,
\]
where $a \, < \, x \, \le \, b$, and
\[
(D^{(K),*}_{b-,g} f)(x) \, = \, 
- \, 
\left(I^{(K)}_{b-,g} \frac{1}{g^{(1)}(u)} \, \frac{d}{du} \, f(u) \right)(x) ,
\]
where $a \, \le \, x \, < \, b$.

%%%%%%%%%%%%%%%%%%%%%%%%%%%%%%%%%%%%%%%%%%%%%%%%%%%%%%%%%%%%%

\subsection{Examples of special cases of parametric GF operators}

Let us give some examples of the parametric GFDs for special forms of the function $g(x)$.

A) If $g(x) \, = \, x$, then the GFDs of the Riemann-Liouville type of function $f(x)$ with respect to function $g(x)$ coincide with the GFDs of the Riemann-Liouville type of function $f(x)$ by
\[
(D^{(K)}_{a+,x} f)(x) \, = \, 
(D^{(K)}_{a+} f)(x) \, = \,
\frac{d}{dx}
\int^x_a K(x \, - \, u) \, f(u) \, du ,
\]
\[
(D^{(K)}_{b-,x} f)(x) \, = \, 
- \, \frac{d}{dx}
\int^b_x K(u \, - \, x) \, f(u) \, du .
\]

B) If $K(x) \, = \, h_{1-\alpha}(x)$, 
then the GFDs of the Riemann-Liouville type of function $f(x)$ with respect to function $g(x)$ coincide with the fractional derivatives of the Riemann-Liouville type of function $f(x)$ with respect to function $g(x)$ \cite{KST,SKM} by
\[
(D^{(h_{1-\alpha})}_{a+,g} f)(x) \, = \, 
\frac{1}{\Gamma(1-\alpha)} 
\frac{1}{g^{(1)}(x)} \, \frac{d}{dx}
\int^x_a (g(x)-g(u))^{-\alpha} \, f(u) \, g^{(1)}(u) \, du ,
\]
\[
(D^{(h_{1-\alpha})}_{b-,g} f)(x) \, = \,
- \, \frac{1}{\Gamma(1-\alpha)} 
\frac{1}{g^{(1)}(x)} \, \frac{d}{dx}
\int^b_x (g(u)-g(x))^{-\alpha} \, f(u) \, g^{(1)}(u) \, du ,
\]
where $\alpha \, \in \, (0,1)$.
For the limit $\alpha \, \to \, 1-$, we have
\[
\lim\limits_{\alpha \to 1-}
(D^{(h_{1-\alpha})}_{a+,g} f)(x) \, = \, 
\frac{f^{(1)}(x)}{g^{(1)}(x)} ,
\]
\[
\lim\limits_{\alpha \to 1-}
(D^{(h_{1-\alpha})}_{b-,g} f)(x) \, = \, 
- \, \frac{f^{(1)}(x)}{g^{(1)}(x)} .
\]
For details see equations 2.5.28 and 2.5.29 in \cite{KST}, p.102.

C) If we consider $g(x) \, = \, ln(x/a)$, we obtain the Hadamard general fractional operators. 
The Hadamard fractional integrals 
$(I^{(h_{\alpha})}_{a+,\ln(x/a)} f)(x)$, 
$(I^{(h_{\alpha})}_{b-,\ln(x/a)} f)(x)$, 
and Hadamard fractional derivatives 
$(D^{(h_{1-\alpha})}_{a+,\ln(x/a)} f)(x)$,
$(D^{(h_{1-\alpha})}_{b-,\ln(x/a)} f)(x)$,
are described in Section 2.7 of \cite{KST}, pp.110-120.

%%%%%%%%%%%%%%%%%%%%%% 

The Hadamard FIs are parametric FIs with respect to the function 
$g(x) \, = \, \ln(x/a)$. 
Therefore the Hadamard FIs are "fractional integral of a function with respect to function $g(x) \, = \, \ln(x/a)$". 
However, the condition of the existence of a continuous derivatives $g^{(1)}(x)$
%%%, which was assumed in \cite{SKM,KST},
does not hold in this case. 
So the Hadamard FIs are need to be independently considered. 
Note that the continuity of $g^{(1)}(x)$ would be satisfied, if we take the lower limit of integration to be equal to $a \, > \, 0$ instead of zero, but then the property of invariance of the integral relative to dilation would be broken \cite{SKM}, p.330. 
Note that the Hadamard-type GFC is proposed 
in \cite{FF-2023} as scale-invariant GFC of the Mellin convolution operators.

D) If we take $g(x) \, = \, x^{\sigma}$, then we get operators that are expressed in terms of Erdelyi-Kober fractional operators
\[
(\, ^{EK}I^{\alpha}_{a+, \sigma, \eta} f)(x) \, = \, 
x^{-\sigma (\alpha+\eta)} \, 
I^{\alpha}_{a+,g}(x^{\sigma \eta} f(x)) ,
\]
\[
(\, ^{EK}D^{\alpha}_{a+, \sigma, \eta} f)(x) \, = \, 
x^{-\sigma \eta} \, 
D^{\alpha}_{a+,g}(x^{\sigma (\alpha + \eta} f(x)) .
\]
The Erdelyi-Kober fractional integrals 
$(\, ^{EK}I^{\alpha}_{a+, \sigma, \eta} f)(x)$ and
the Erdelyi-Kober fractional derivatives $(\, ^{EK}D^{\alpha}_{a+, \sigma, \eta} f)(x)$ see in Section 2.6 of \cite{KST}, pp.105-110.
Note that in general form with 3 parameters $(\alpha, \sigma, \eta)$ it appeared first in the book of \cite{Sneddon1} and his survey in \cite{Sneddon1}\cite{Sneddon2} "The use in mathematical analysis of E-K operators and some of their applications", 37-79). Then detailed theory of E-K operators was developed in Ch. 2 of \cite{Kiryakova} and also in book of \cite{YakubovichLuchko}.

%%%%%%%%%%%%%%%%%%%%%%%%%%%%%%%%%%%%%%%%%%%%%%%%%%%%%%%%%%%%%
%%%%%%%%%%%%%%%%%%%%%%%%%%%%%%%%%%%%%%%%%%%%%%%%%%%%%%%%%%%%%
%%%%%%%%%%%%%%%%%%%%%%%%%%%%%%%%%%%%%%%%%%%%%%%%%%%%%%%%%%%%%

\subsection{GF operators on $(a,b)$ and $[0,b-a]$}

Let us note relations of the GFI on the finite interval $(a,b)$, which is proposed in \cite{Luchko2023},
and the GFI on $[0,b-a]$, which is proposed in \cite{Luchko2021-1}.

In Luchko work \cite{Luchko2023}, 
the GFI on the finite interval $(a,b)$ is
\[
(I^{(M)}_{a+} f)(x) \, = \, 
\int^x_a M(x-u) \, f(u) \, du .
\]

Let us prove the following statement.

\begin{Proposition}

The left-sided GFI on the finite interval $(a,b)$ can be expressed via 
the GFI on $[0,b-1]$ by the equation
\begin{equation}
(I^{(M)}_{a+} f)(x) \, = \, 
(Q_{g} \, (I^{(M)}_{0+} \, Q^{-1}_g \, f))(x) ,
\end{equation}
if $f(x) \, \in \,C_{-1}(a,b]$, 
where the substitution operator 
$Q_g$ with $g(x) \, = \, x \, - \, a$.

The right-sided GFI on the finite interval $(a,b)$ can be expressed via 
the GFI on $[0,b-1]$ by the equation
\begin{equation}
(I^{(M)}_{b-} f)(x) \, = \, 
(Q_{g} \, (I^{(M)}_{0+} \, Q^{-1}_g \, f))(x) ,
\end{equation}
if $f(x) \, \in \,C_{-1}[a,b)$, 
where the substitution operator 
$Q_g$ with $g(x) \, = \, b \, - \, x$.

\end{Proposition}

\begin{proof}

Using the variable $\xi \, = \, u \, - \, a$, we get
\[
(I^{(M)}_{a+} f)(x) \, = \, 
\int^x_a M(x-u) \, f(u) \, du \, = \, 
\]
\[
\int^{x-a}_0 M(x-(\xi+a)) \, f(\xi \, + \, a) \, d \xi 
\, = \, 
\int^{x-a}_0 M((x-a)-\xi) \, (Q^{-1}_g \, f)(\xi) \, d \xi 
\, = \, 
\]
\[
Q_g \, 
\int^{x}_0 M(x-\xi) \, (Q^{-1}_g \, f)(\xi) \, d \xi 
\, = \, 
Q_{g} \, ( I^{(M)}_{0+} \, Q^{-1}_g \, f)(x) ,
\]
where $g(x)=x-a$

Using the variable $\xi \, = \, b \, - \, u$, we get
\[
(I^{(M)}_{b-} f)(x) \, = \, 
\int^b_x M(u-x) \, f(u) \, du \, = \, 
\]
\[
\int^{b-x}_0 M((b-\xi)-x) \, f(b \, - \, \xi) \, d \xi 
\, = \, 
\int^{b-x}_0 M((b-x)-\xi) \, (Q^{-1}_g \, f)(\xi) \, d \xi 
\, = \, 
\]
\[
Q_g \, 
\int^{x}_0 M(x-\xi) \, (Q^{-1}_g \, f)(\xi) \, d \xi 
\, = \, 
Q_{g} \, ( I^{(M)}_{0+} \, Q^{-1}_g \, f)(x) .
\]
where $g(x)=b-x$

\end{proof}

Therefore, some properties of the GFIs on finite interval $(a,b)$, that are described in \cite{Luchko2023}, follow from the corresponding properties of the GFIs that are described in \cite{Luchko2021-1}. 

%%%%%%%%%%%%%%%%%%%%%%%%%%%%%%%%%%%%%%%%%%%%%%%%%%%%%%%%%%%%%
%%%%%%%%%%%%%%%%%%%%%%%%%%%%%%%%%%%%%%%%%%%%%%%%%%%%%%%%%%%%%
%%%%%%%%%%%%%%%%%%%%%%%%%%%%%%%%%%%%%%%%%%%%%%%%%%%%%%%%%%%%%

\subsection{Parametric GF integrals via GF integrals on $(a,b)$}

Let us prove the following statement.

\begin{Proposition} \label{PGFI-GFIab}
Let $g(x)$ be an increasing and positive monotone function on $[a, b)$ 
with $ - \, \infty \, < \, a \, < \, b \, < \, +\infty$, 
having a continuous first-order derivative $g^{(1)}(x) \, = \, dg(x)/dx$ for all $(a,b)$.

If $g^{(1)}(x) \, \ne \, 0$ for all $x \in (a,b)$,
the left-sided and right-sided parametric GFIs of $f(x)$ with respect to $g(x)$ can be expressed via 
the GFIs on the interval $(a,b)$ by the equations
\begin{equation} \label{I-QIQ-l}
(I^{(M)}_{a+,g} f)(x) \, = \, 
Q_{g} \, (I^{(M)}_{g(a)+} \, Q^{-1}_g \, f)(x) ,
\end{equation}
if $f(x) \, \in \,C_{-1,g}(a,b]$, and
\begin{equation} \label{I-QIQ-r}
(I^{(M)}_{b-,g} f)(x) \, = \, 
Q_{g} \, (I^{(M)}_{g(b)-} \, Q^{-1}_g \, f)(x) ,
\end{equation}
if $f(x) \, \in \, C_{-1,g}[a,b)$, 
where $Q_g$ is the substitution operator.

\end{Proposition}

\begin{proof}

Using the definition of the GFIs of $f(x)$ with respect to $g(x)$
and the variable $\xi \, = \, g(u)$, we get
\[
(I^{(M)}_{a+,g} f)(x) \, = \, 
\int^x_a M(g(x)-g(u)) \, f(u) \, g^{(1)}(u) \, du 
\, = \, 
\]
\[
\int^{g(x)}_{g(a)} M(g(x) \, - \, \xi) \, 
f(g^{-1}(\xi)) \, d \xi 
\, = \, 
\]
\[
\int^{g(x)}_{g(a)} M(g(x) \, - \, \xi) \, 
(Q^{-1}_g f(\xi)) \, d \xi 
\, = \, 
\]
\[
Q_{g} \, \int^{x}_{g(a)} M(x \, - \, \xi) \, 
(Q^{-1}_g f(\xi)) \, d \xi 
\, = \, 
\]
\[
Q_{g} \, (I^{(M)}_{g(a)+} Q^{-1}_g f)(x) ,
\]
which was to be proved. 

For right-sided GFIs of $f(x)$ with respect to $g(x)$

\[
(I^{(M)}_{b-,g} f)(x) \, = \, 
\int^b_x M(g(u)-g(x)) \, f(u) \, g^{(1)}(u) \, du 
\, = \, 
\]
\[
\int^{g(b)}_{g(x)} M(\xi \, - \, g(x)) \, 
f(g^{-1}(\xi)) \, d \xi 
\, = \, 
\]
\[
\int^{g(b)}_{g(x)}M(\xi \, - \, g(x)) \, 
(Q^{-1}_g f(\xi)) \, d \xi 
\, = \, 
\]
\[
Q_{g} \, \int^{g(b)}_{x} M(\xi \, - \, x) \, 
(Q^{-1}_g f(\xi)) \, d \xi 
\, = \, 
\]
\[
Q_{g} \, (I^{(M)}_{g(b)-} Q^{-1}_g f)(x) ,
\]
which was to be proved.

\end{proof}

Therefore, some properties of 
the GFIs of $f(x)$ with respect to $g(x)$
follow from the corresponding properties of 
the GFIs on finite interval $(a,b)$ that are described in \cite{Luchko2023}. 

\begin{Proposition}[Semi-group property of parametric GFIs]

Let kernel pairs $(M_1(x), \, K_1(x))$ and $(M_2(x), \, K_2(x))$ belong to the Luchko set ${\cal L}_{1}(\mathbb{R}_{+})$.
Let $g(x)$ be an increasing and positive monotone function on $[a, b)$ 
with $ - \, \infty \, < \, a \, < \, b \, < \, +\infty$, 
having a nonzero continuous first-order derivative $g^{(1)}(x) \, \ne \, 0$ for all $(a,b)$.

Then, the left-sided and right-sided parametric GFIs of the satisfy the equalities
\begin{equation}
(I^{(M_1)}_{a+,g} \, I^{(M_2)}_{a+,g} f)(x) \, = \, (I^{(M_1*M_2)}_{a+,g} f)(x)
\end{equation}
if $f(x) \, \in \, C_{-1,g}(a,b]$, 
i.e. $(Q^{-1}_g f)(x) \, \in \, C_{-1}(a,b]$, and
\begin{equation}
(I^{(M_1)}_{b-,g} \, I^{(M_2)}_{b-,g} f)(x) \, = \, (I^{(M_1*M_2)}_{b-,g} f)(x)
\end{equation}
if $f(x) \, \in \, C_{-1,g}[a,b)$, 
i.e. $(Q^{-1}_g f)(x) \, \in \, C_{-1}[a,b)$. 
\end{Proposition}

\begin{proof}
Let us prove the inequality of the left-sided parametric GFI.

In the proof one can use the fact that
$(M_1 \, * \, M_2)(x) \, \in \, C_{-1}(0,b-a]$,
if $M_1(x) \, \in \, C_{-1}(0,b-a]$ and
$M_2(x) \, \in \, C_{-1}(0,b-a]$
(see \cite{Luchko2023}, p.6).

Using the representation of the parametric GFIs via GFIs on the interval $(a,b)$ 
(see Proposition \ref{PGFI-GFIab}), 
and Proposition 3 of \cite{Luchko2023}, pp.5-6, one can get
\[
(I^{(M_1)}_{a+,g} \, I^{(M_2)}_{a+,g} f)(x) \, = \, 
\left( (Q_g \, I^{(M_1)}_{g(a)+} \, Q^{-1}_g) \, 
(Q_g \, I^{(M_2)}_{g(a)+} \, Q^{-1}_g) \, f \right)(x) \, = \, 
\]
\[
\left( (Q_g \, (I^{(M_1)}_{g(a)+} \, 
I^{(M_2)}_{g(a)+}) \, Q^{-1}_g) \, f \right)(x) \, = \, 
\left( (Q_g \, (I^{(M_1 * M_2)}_{g(a)+} \, 
\, Q^{-1}_g) \, f \right)(x) \, = \, 
 (I^{(M_1*M_2)}_{a+,g} f)(x) ,
\]
which was to be proved. 

The property of the right-sided GFD of the RL type is proved similarly.
Q.E.D.
\end{proof}

%%%%%%%%%%%%%%%%%%%%%%%%%%%%%%%%%%%%%%%%%%%%%%%%%%%%%%%%%%%%%
%%%%%%%%%%%%%%%%%%%%%%%%%%%%%%%%%%%%%%%%%%%%%%%%%%%%%%%%%%%%%
%%%%%%%%%%%%%%%%%%%%%%%%%%%%%%%%%%%%%%%%%%%%%%%%%%%%%%%%%%%%%

\subsection{Parametric GF derivatives via GF derivatives on $(a,b)$}

Similar properties of the parametric GFIs
exist for the parametric GFDs of the Riemann-Liouville type in the form.
%%%of $f(x)$ with respect to $g(x)$ in the form

\begin{Proposition} \label{D-QDQ-RL}
Let kernel pair $(M(x), \, K(x))$ belongs to the Luchko set ${\cal L}_{1}(\mathbb{R}_{+})$.
Let $g(x)$ be an increasing and positive monotone function on $[a, b)$ 
with $ - \, \infty \, < \, a \, < \, b \, < \, +\infty$, 
having a nonzero continuous first-order derivative $g^{(1)}(x) \, \ne \, 0$ for all $(a,b)$.

Then, the left-sided and right-sided parametric GFDs of the Riemann-Liouville type satisfy the equalities
\begin{equation}
(D^{(K)}_{a+,g} f)(x) \, = \, 
Q_{g} \, (D^{(K)}_{g(a)+} \, Q^{-1}_g \, f)(x) ,
\end{equation}
if $f(x) \, \in \, C^{1}_{-1,g}(a,b]$, i.e. 
$f(g^{-1}(x)) \, \in \, C^{1}_{-1}(a,b]$, and 
\begin{equation}
(D^{(K)}_{b-,g} f)(x) \, = \, 
Q_{g} \, (D^{(K)}_{g(b)-} \, Q^{-1}_g \, f)(x) .
\end{equation}
if $f(x) \, \in \, C^{1}_{-1,g}[a,b)$, i.e. 
$f(g^{-1}(x)) \, \in \, C^{1}_{-1}[a,b)$. 
\end{Proposition}

%%% PROOF for RL
\begin{proof}
Let us prove the inequality of the left-sided parametric GFD.

Using the definition of the left-sided parametric GFD of the RL type and the substitution operator
\[
(D^{(K)}_{a+,g} f)(x) \, = \, 
\left(\frac{1}{g^{(1)}(x)} \frac{d}{dx} \right) \, 
(I^{(K)}_{a+,g} f)(x) \, = \, 
\]
\[
\left(\frac{1}{g^{(1)}(x)} \frac{d}{dx} \right) \, 
Q_g \, (I^{(K)}_{g(a)+} \, Q^{-1}_g \, f)(x) \, = \, 
\]
\[
Q_g \, \frac{d}{d x} \, (I^{(K)}_{g(a)+} \, Q^{-1}_g \, f)(x) \, = \, 
Q_g \, \left( D^{(K)}_{g(a)+} \, Q^{-1}_g \, f \right) (x) ,
\]
which was to be proved. 

The property of the right-sided GFD of the RL type is proved similarly.
Q.E.D.
\end{proof}

Note that for usual fractional derivatives of the Riemann-Liouville type
these properties are presented in \cite{SKM}, pp.326-327.

Similar properties exist for the parametric GFDs of the Caputo type in the form.

\begin{Proposition} \label{D-QDQ-C}
Let kernel pair $(M(x), \, K(x))$ belongs to the Luchko set ${\cal L}_{1}(\mathbb{R}_{+})$.
Let $g(x)$ be an increasing and positive monotone function on $[a, b)$ 
with $ - \, \infty \, < \, a \, < \, b \, < \, +\infty$, 
having a nonzero continuous first-order derivative $g^{(1)}(x) \, \ne \, 0$ for all $(a,b)$.

Then, the left-sided and right-sided parametric GFDs of the Caputo type satisfy the equalities
\begin{equation}
({}D^{(K),*}_{a+,g} f)(x) \, = \, 
Q_{g} \, (D^{(K),*}_{g(a)+} \, Q^{-1}_g \, f)(x) ,
\end{equation}
if $f(x) \, \in \, C^{1}_{-1,g}(a,b]$, i.e. 
$f(g^{-1}(x)) \, \in \, C^{1}_{-1}(a,b]$, and 
\begin{equation}
(D^{(K),*}_{b-,g} f)(x) \, = \, 
Q_{g} \, (D^{(K),*}_{g(b)-} \, Q^{-1}_g \, f)(x) ,
\end{equation}
if $f(x) \, \in \, C^{1}_{-1,g}[a,b)$, i.e. 
$f(g^{-1}(x)) \, \in \, C^{1}_{-1}[a,b)$. 
\end{Proposition}

%%% PROOF for Caputo
\begin{proof}
Let us prove the inequality of the left-sided parametric GFD.

Using the definition of the left-sided parametric GFD of the Caputo type and the substitution operator
\[
(D^{(K),*}_{a+,g} \, f)(x) \, = \, 
\left(I^{(K)}_{a+,g} \, 
\left(\frac{1}{g^{(1)}(u)} \frac{d}{du} \right)\, f(u) \right)(x) \, = \, 
\]
\[
\left(Q_g \, I^{(K)}_{g(a)+} \, Q^{-1}_g
\left(\frac{1}{g^{(1)}(u)} \frac{d}{du} \right) \, 
f(u) \right)(x) \, = \, 
\]
\[
\left(Q_g \, I^{(K)}_{g(a)+} \, Q^{-1}_g
\left(\frac{1}{g^{(1)}(u)} \frac{d}{du} \right) \, 
Q_g \, Q^{-1}_g \, f(u) \right)(x) \, = \, 
\]
\[
\left(Q_g \, I^{(K)}_{g(a)+} \, Q^{-1}_g \, Q_g \,
\left( \frac{d}{du} \right) Q^{-1}_g \, f(u) \right)(x) \, = \, 
\]
\[
\left(Q_g \, I^{(K)}_{g(a)+} \, \left( \frac{d}{du} \right) \, Q^{-1}_g \, f(u) \right)(x) \, = \, 
\]
\[
\left(Q_g \, D^{(K),*}_{g(a)+} \, Q^{-1}_g \, f \right)(x) ,
\]
which was to be proved. 

The property of the right-sided GFD of the Caputo type is proved similarly.
Q.E.D.
\end{proof}

The parametric GFD of the RL type can be writted via 
the parametric GFD of the Caputo type.

\begin{Proposition}[Relation between parametric GFDs of RL and Caputo types] 
\label{RL-Caputo}

Let kernel pair $(M(x), \, K(x))$ belongs to the Luchko set ${\cal L}_{1}(\mathbb{R}_{+})$.
Let $g(x)$ be an increasing and positive monotone function on $[a, b)$ 
with $ - \, \infty \, < \, a \, < \, b \, < \, +\infty$, 
having a nonzero continuous first-order derivative $g^{(1)}(x) \, \ne \, 0$ for all $(a,b)$.

Then, the left-sided and right-sided parametric GFDs of the Riemann-Liouville and Caputo types satisfy the equalities
\begin{equation}
(D^{(K)}_{a+,g} \, f)(x) \, = \, 
(D^{(K),*}_{a+,g} \, f)(x) \, + \, 
K(g(x) \, - \, g(a)) \, f(g(a)) , 
\end{equation}
if $f(x) \, \in \, C^{1}_{-1,g}(a,b]$, i.e. 
$f(g^{-1}(x)) \, \in \, C^{1}_{-1}(a,b]$, and 
\begin{equation}
(D^{(K)}_{b-,g} \, f)(x) \, = \, 
(D^{(K),*}_{b-,g} \, f)(x) \, + \, 
K(g(b) \, - \, g(x)) \, f(g(b)) , 
\end{equation}
if $f(x) \, \in \, C^{1}_{-1,g}[a,b)$, i.e. 
$f(g^{-1}(x)) \, \in \, C^{1}_{-1}[a,b)$. 
\end{Proposition}

\begin{proof}
Let us prove the inequality of the left-sided parametric GFD.

Using Proposition \ref{D-QDQ-RL} of this paper and Proposition 4 of Luchko paper \cite{Luchko2023}, one can get
\[
(D^{(K)}_{a+,g} \, f)(x) \, = \, 
Q_g \, \left(D^{(K)}_{g(a)+} \, Q^{-1}_g \, f \right)(x) \, = \, 
\]
\[
Q_g \, \left( 
(D^{(K),*}_{g(a)+} \, Q^{-1}_g \, f)(x) 
\, + \, K(x \, - \, g(a)) \, (Q^{-1}_g \, f)(g(a))
\right) \, = \, 
\]
\[
Q_g \, \left(D^{(K),*}_{g(a)+} \, Q^{-1}_g \, f \right)(x) \, + \, 
Q_g \, \left( K(x \, - \, g(a)) \, (Q^{-1}_g \, f)(g(a) \right) 
\, = \, 
\]
\[
(D^{(K),*}_{a+,g} f)(x)\, + \, 
K(g(x) \, - \, g(a)) \, f(g(a) ,
\]
which was to be proved. 

For the right-sided GFD, the identity is proved similarly.
Q.E.D.
\end{proof}

%%%%%%%%%%%%%%%%%%%%%%%%%%%%%%%%%%%%%%%%%%%%%%%%%%%%%%%%%%%%%
%%%%%%%%%%%%%%%%%%%%%%%%%%%%%%%%%%%%%%%%%%%%%%%%%%%%%%%%%%%%%
%%%%%%%%%%%%%%%%%%%%%%%%%%%%%%%%%%%%%%%%%%%%%%%%%%%%%%%%%%%%%

%%% \newpage

%%%%%%%%%%%%%%%%%%%%%%%%%%%%%%%%%%%%%%%%%%%%%%%%%%%%%%%%%%%%%

\section{Fundamental Theorems of Parametric GFC}

%%%%%%%%%%%%%%%%%%%%%%%%%%%%%%%%%%%%%%%%%%%%%%%%%%%%%%%%%%%%%

The properties of GFDs of $f(x)$ with respect to $g(x)$
can be proved by using the properties of GFIs of $f(x)$ with respect to $g(x)$.

Using the properties of the parametric GFIs in form of Proposition \ref{PGFI-GFIab} and the properties of the substitution operators $Q_g$, $Q^{-1}_g$, one can prove the following representations.

The left-sided and right-sided parametric GFDs of the Riemann-Liouville type can be written as
\[
(D^{(K)}_{a+,g} f)(x) \, = \, 
\frac{1}{g^{(1)}(x)} \frac{d}{dx} \int^x_a K(g(x) \, - \, g(u)) \, f(u) \, g^{(1)}(u) \, du 
\, = \, 
\]
\begin{equation}
\frac{1}{g^{(1)}(x)} \frac{d}{dx} (I^{(K)}_{a+,g} f)(x) \, = \, 
\frac{1}{g^{(1)}(x)} \frac{d}{dx} (Q_g \, I^{(K)}_{g(a)+} \, Q^{-1}_g\, f)(x) ,
\end{equation}
where $x \, > \, a$, and
\[
(D^{(K)}_{b-,g} f)(x) \, = \, 
\frac{1}{g^{(1)}(x)} \frac{d}{dx} \int^b_x K(g(u) \, - \, g(x)) \, f(u) \, g^{(1)}(u) \, du \, = \, 
\]
\begin{equation}
\frac{1}{g^{(1)}(x)} \frac{d}{dx} (I^{(K)}_{b-,g} f)(x) \, = \, 
\frac{1}{g^{(1)}(x)} \frac{d}{dx} (Q_g \, I^{(K)}_{g(b)-} \, Q^{-1}_g\, f)(x) ,
\end{equation}
where $x \, < \, b$.

The left-sided parametric GFD of the Caputo type can be written as
\[
(D^{(K),*}_{a+,g} f)(x) \, = \, 
\int^x_a K(g(x) \, - \, g(u)) \, f^{(1)}(u) \, du \, = \,
\]
\[
\int^x_a K(g(x) \, - \, g(u)) \, g^{(1)}(u) \left( \frac{1}{g^{(1)}(u)} \frac{d}{du} f(u) \right) \, du 
\, = \, \frac{1}{g^{(1)}(x)} \frac{d}{dx} (I^{(K)}_{a+,g} f)(x) \, = \, 
\]
\begin{equation}
\left( I^{(K)}_{a+,g} \left( \frac{1}{g^{(1)}(u)} f^{(1)}(u) \right)\right)(x) \, = \, 
Q_{g} \, \left( I^{(K)}_{g(a)+} Q^{-1}_g \, \left( \frac{1}{g^{(1)}(u)} f^{(1)}(u) \right) \right)(x) ,
\end{equation}
where $x \, > \, a$, and
\[
(D^{(K),*}_{b-,g} f)(x) \, = \, 
\int^b_x K(g(u) \, - \, g(x)) \, f^{(1)}(u) \, du \, = \,
\]
\[
\int^b_x K(g(u) \, - \, g(x)) \, g^{(1)}(u) \left( \frac{1}{g^{(1)}(u)} \frac{d}{du} f(u) \right) \, du 
\, = \, \frac{1}{g^{(1)}(x)} \frac{d}{dx} (I^{(K)}_{b-,g} f)(x) \, = \, 
\]
\begin{equation}
\left( I^{(K)}_{b-,g} \left( \frac{1}{g^{(1)}(u)} f^{(1)}(u) \right)\right)(x) \, = \, 
Q_{g} \, \left( I^{(K)}_{g(b)-} Q^{-1}_g \, 
\left( \frac{1}{g^{(1)}(u)} f^{(1)}(u) \right) \right)(x) ,
\end{equation}
where $x \, < \, b$.

The right-sided parametric GFDs of the Riemann-Liouville and Caputo types are defined 

%%%%%%%%%%%%%%%%%%%%%%%%%%%%%%%%%%%%%%%%%%%%%%%%%%%%%%%%%%%%%
%%%%%%%%%%%%%%%%%%%%%%%%%%%%%%%%%%%%%%%%%%%%%%%%%%%%%%%%%%%%%
%%%%%%%%%%%%%%%%%%%%%%%%%%%%%%%%%%%%%%%%%%%%%%%%%%%%%%%%%%%%%

\subsection{First Fundamental Theorems of Parametric GFC}

%%%%%%%%%%%%%%%%%%%%%%%%%%%%%%%%%%%%%%%%%%%%%%%%%%%%%%%%%%%%%

Let us prove the first fundamental theorems of parametric GFC for the GFDs of Riemann-Liouville types.
Note that the condition 
$f(x) \, \in \, C^{1}_{-1,g}(\Omega)$ 
means that
$f(g^{-1}(x)) \, \in \, C^{1}_{-1}(\Omega)$
for finite interval $\Omega \subset \mathbb{R}$.

\begin{Theorem}[First fundamental theorem for the parametric GFD of the Riemann-Liouville type]
\label{FT-1-RL} 

Let kernel pair $(M(x), \, K(x))$ belongs to the Luchko set ${\cal L}_{1}(\mathbb{R}_{+})$ and
$ - \, \infty \, < \, a \, < \, b \, < \, +\infty$.

Then, the parametric GFI and the parametric GFD of the RL type satisfy the equalities
\[
(D^{(K)}_{a+,g} \, I^{(M)}_{a+,g} f)(x) \, = \, f(x) ,
\]
if $f(x) \, \in \, C^{1}_{-1,g}(a,b]$ 
and $g(x) \, \in \, {\cal G}(a,b]$, and 
\[
(D^{(K)}_{b-,g} \, I^{(M)}_{b-,g} f)(x) \, = \, f(x) ,
\]
if $f(x) \, \in \, C^{1}_{-1,g}[a,b)$ 
and $g(x) \, \in \, {\cal G}[a,b)$.
\end{Theorem}

%%% PROOF 1-FT-RL
\begin{proof}
Using the definition of the GFD of the RL type and the substitution operator
\[
(D^{(K)}_{a+,g} \, I^{(M)}_{a+,g} f)(x) \, = \, 
\left( \left(\frac{1}{g^{(1)}(x)} \frac{d}{dx} \right) \, 
I^{(K)}_{a+,g} \,
\left( I^{(M)}_{a+,g} \,f \right) \right)(x) \, = \, 
\]
\[
\left( \left(\frac{1}{g^{(1)}(x)} \frac{d}{dx} \right)
\, Q_g \, I^{(K)}_{g(a)+} \, Q^{-1}_g \, Q_g \, 
\left( I^{(M)}_{g(a)+} \, Q^{-1}_g \, f \right) \right)(x) 
\, = \, 
\]
\[
\left( \left(\frac{1}{g^{(1)}(x)} \frac{d}{dx} \right) 
\, Q_g \,I^{(K)}_{g(a)+} \,
\left( I^{(M)}_{g(a)+} \, Q^{-1}_g \, f \right) \right)(x) 
\, = \, 
\]
\[
\left( 
Q_g \, \left(\frac{d}{dx} \right) \, I^{(K)}_{g(a)+} \, 
\left( I^{(M)}_{g(a)+} \, Q^{-1}_g \, f \right) \right)(x) 
\, = \, 
\]
\[
Q_g \, \left( D^{(K)}_{g(a)+} \, 
\left( I^{(M)}_{g(a)+} \, Q^{-1}_g \, f \right) \right)(x) .
\]
Using the first FT of GFC \cite{Luchko2023}, we get
\[
(D^{(K)}_{a+,g} \, I^{(M)}_{a+,g} f)(x) \, = \, 
Q_g \, Q^{-1}_g \, f(x) \, = \, f(x) ,
\]
which was to be proved. 

For the right-sided parametric GFD, the identity is proved similarly.
Q.E.D.
\end{proof}

For the parametric GFD of the Caputo type and parametric GFIs, 
the following statement holds.

\begin{Theorem}[First fundamental theorem for the GFDs of the Caputo type]
\label{FT-1-C}

Let kernel pair $(M(x), \, K(x))$ belongs to the Luchko set ${\cal L}_{1}(\mathbb{R}_{+})$ and
$ - \, \infty \, < \, a \, < \, b \, < \, +\infty$.

Then, the parametric GFI and the parametric GFD of the Caputo type satisfy the equalities
\[
(D^{(K),*}_{a+,g} \, I^{(M)}_{a+,g} f)(x) \, = \, f(x) ,
\]
if $f(x) \, \in \, C^{1}_{-1,g}(a,b]$ 
and $g(x) \, \in \, {\cal G}(a,b]$, and 
\[
(D^{(K),*}_{b-,g} \, I^{(M)}_{b-,g} f)(x) \, = \, f(x) ,
\]
if $f(x) \, \in \, C^{1}_{-1,g}[a,b)$ 
and $g(x) \, \in \, {\cal G}[a,b)$. 
\end{Theorem}

%%% PROOF 1-FT-C
\begin{proof}
Using the definition of the GFD of the Caputo type and the substitution operator
\[
(D^{(K),*}_{a+,g} \, I^{(M)}_{a+,g} f)(x) \, = \, 
\left(I^{(K)}_{a+,g} \, 
\left(\frac{1}{g^{(1)}(u)} \frac{d}{du} \right)\, (I^{(M)}_{a+,g} f)(u) \right)(x) \, = \, 
\]
\[
\left(Q_g \, I^{(K)}_{g(a)+} \, Q^{-1}_g
\left(\frac{1}{g^{(1)}(u)} \frac{d}{du} \right) \, 
Q_g \, (I^{(M)}_{g(a)+} Q^{-1}_g \, f)(u) \right)(x) .
\]
Using property \eqref{Property-DQ-QD} and the first FT of GFC \cite{Luchko2023}, we get
\[
(D^{(K),*}_{a+,g} \, I^{(M)}_{a+,g} f)(x) \, = \, 
\left(Q_g \, I^{(K)}_{g(a)+} \, Q^{-1}_g \, Q_g \,
\left( \frac{d}{du} \right) (I^{(M)}_{g(a)+} Q^{-1}_g \, f)(u) \right)(x) \, = \, 
\]
\[
\left(Q_g \, I^{(K)}_{g(a)+} \, 
\left( \frac{d}{du} \right) (I^{(M)}_{g(a)+} Q^{-1}_g \, f)(u) \right)(x) \, = \, 
\]
\[
\left(Q_g \, D^{(K),*}_{g(a)+} \,
(I^{(M)}_{g(a)+} Q^{-1}_g \, f) \right)(x) \, = \, 
Q_g \, Q^{-1}_g \, f(x) \, = \, f(x) ,
\]
which was to be proved. 

For the right-sided parametric GFD, the identity is proved similarly.
Q.E.D.
\end{proof}

%%%%%%%%%%%%%%%%%%%%%%%%%%%%%%%%%%%%%%%%%%%%%%%%%%%%%%%%%%%%%
%%%%%%%%%%%%%%%%%%%%%%%%%%%%%%%%%%%%%%%%%%%%%%%%%%%%%%%%%%%%%
%%%%%%%%%%%%%%%%%%%%%%%%%%%%%%%%%%%%%%%%%%%%%%%%%%%%%%%%%%%%%

\subsection{Second Fundamental Theorems of Parametric GFC}

%%%%%%%%%%%%%%%%%%%%%%%%%%%%%%%%%%%%%%%%%%%%%%%%%%%%%%%%%%%%%

Let us prove the second fundamental theorems of parametric GFC for the GFDs of Riemann-Liouville types.

\begin{Theorem}[Second fundamental theorem for the GFDs of the Riemann-Liouville type]
\label{FT-2-RL} 

Let kernel pair $(M(x), \, K(x))$ belongs to the Luchko set ${\cal L}_{1}(\mathbb{R}_{+})$ and 
$ - \, \infty \, < \, a \, < \, b \, < \, +\infty$. 

Then, the parametric GFI and the parametric GFD of the Riemann-Liouville type satisfy the equalities
\[
(I^{(M)}_{a+,g} \, D^{(K)}_{a+,g} f)(x) \, = \, f(x) ,
\]
if $f(x) \, \in \, C^{1}_{-1,g}(a,b]$ 
and $g(x) \, \in \, {\cal G}(a,b]$, and 
\[
(I^{(M)}_{b-,g} \, D^{(K)}_{b-,g} f)(x) \, = \, f(x) ,
\]
if $f(x) \, \in \, C^{1}_{-1,g}[a,b)$ 
and $g(x) \, \in \, {\cal G}[a,b)$.
\end{Theorem}

%%% PROOF 2-FT-RL
\begin{proof}
Using the definition of the parametric GFD of the RL type and the substitution operator, we get
\[
(I^{(M)}_{a+,g} \, D^{(K)}_{a+,g} f)(x) \, = \, 
\left(I^{(M)}_{a+,g} \, \left(\frac{1}{g^{(1)}(u)} \frac{d}{du} \, (I^{(K)}_{a+,g} f \right) \right)(x) \, = \, 
\]
\[
\left(Q_g \, I^{(M)}_{g(a)+} \, Q^{-1}_g 
\left(\frac{1}{g^{(1)}(u)} \frac{d}{du} \right) \, 
Q_g \, (I^{(K)}_{g(a)+} \, Q^{-1}_g \, f \right)(x) . 
\]
Using property \eqref{Property-DQ-QD} and the second FT of GFC \cite{Luchko2023}, we get
\[
(I^{(M)}_{a+,g} \, D^{(K)}_{a+,g} f)(x) \, = \, 
\left(Q_g \, I^{(M)}_{g(a)+} \, Q^{-1}_g \, 
Q_g \, \frac{d}{du} \, (I^{(K)}_{g(a)+} \, Q^{-1}_g \, f \right)(u) \, = \, 
\]
\[
\left(Q_g \, I^{(M)}_{g(a)+} \, \frac{d}{du} \, 
\left(I^{(K)}_{g(a)+} \, Q^{-1}_g \, f \right) \right)(u) 
\, = \, 
\left(Q_g \, I^{(M)}_{g(a)+} \, 
\left( D^{(K)}_{g(a)+} \, Q^{-1}_g \, f \right) \right) (u) 
\, = \, 
\]
\[
Q_g \, \left( (Q^{-1}_g \, f)(x) \, - \, 
(Q^{-1}_g \, f)(x) \right) \, = \, f(x) ,
\]
which was to be proved. 

For the right-sided parametric GFD, the identity is proved similarly.
Q.E.D.
\end{proof}

For the parametric GFD of the Caputo type and parametric GFIs, 
the following statement holds.

\begin{Theorem}[Second fundamental theorem for the GFDs of the Caputo type] 
\label{FT-2-C}

Let kernel pair $(M(x), \, K(x))$ belongs to the Luchko set ${\cal L}_{1}(\mathbb{R}_{+})$ and
$ - \, \infty \, < \, a \, < \, b \, < \, +\infty$.

Then, the parametric GFI and the parametric GFD of the Caputo type satisfy the equality
\[
(I^{(M)}_{a+,g} \, D^{(K),*}_{a+,g} \, f)(x) 
\, = \, f(x) \, - \, f(a) ,
\]
if $f(x) \, \in \, C^{1}_{-1,g}(a,b]$ 
and $g(x) \, \in \, {\cal G}(a,b]$, and 
\[
(I^{(M)}_{b-,g} \, D^{(K),*}_{b-,g} \, f)(x) 
\, = \, f(b) \, - \, f(x) ,
\]
if $f(x) \, \in \, C^{1}_{-1,g}[a,b)$ 
and $g(x) \, \in \, {\cal G}[a,b)$.
\end{Theorem}

%%% PROOF 2-FT-C
\begin{proof}
Using the definition of the parametric GFD of the Caputo type and 
property \eqref{Property-DQ-QD} of the substitution operator, we get
\[
(I^{(M)}_{a+,g} \, D^{(K),*}_{a+,g} \, f)(x) \, = \, 
\left( I^{(M)}_{a+,g} \,
I^{(K)}_{a+,g} \, 
\left(\frac{1}{g^{(1)}(u)} \frac{d}{du} \right)\, 
f(u) \right)(x) \, = \, 
\]
\[
\left( Q_g \, I^{(M)}_{g(a)+} \, Q^{-1}_g \, Q_g \, 
I^{(K)}_{g(a)+} \, Q^{-1}_g \, 
\left(\frac{1}{g^{(1)}(u)} \frac{d}{du} \right)\, 
f(u) \right)(x) \, = \, 
\]
\[
\left( Q_g \, I^{(M)}_{g(a)+} \, 
I^{(K)}_{g(a)+} \, Q^{-1}_g \, 
\left(\frac{1}{g^{(1)}(u)} \frac{d}{du} \right)\, 
f(u) \right)(x) \, = \, 
\]
\[
\left( Q_g \, I^{(M)}_{g(a)+} \, 
I^{(K)}_{g(a)+} \, Q^{-1}_g \, 
\left(\frac{1}{g^{(1)}(u)} \frac{d}{du} \right)
\, Q_g \, Q^{-1}_g \, 
f(u) \right)(x) \, = \, 
\]
\[
\left( Q_g \, I^{(M)}_{g(a)+} \, 
I^{(K)}_{g(a)+} \, Q^{-1}_g \, Q_g \, 
\left(\frac{d}{du} \right)\, Q^{-1}_g \, 
f(u) \right)(x) \, = \, 
\]
\[
\left( Q_g \, I^{(M)}_{g(a)+} \, 
I^{(K)}_{g(a)+} \, \left(\frac{d}{du} \right)\, Q^{-1}_g \, 
f(u) \right)(x) \, = \, 
\left( Q_g \, I^{(M)}_{g(a)+} \, 
D^{(K)}_{g(a)+} \, Q^{-1}_g \, f(u) \right)(x) . 
\]
Using the second FT of GFC \cite{Luchko2023} and
\[
Q_g \, f(a) \, = \, f(a) , \quad
Q^{-1}_g \, f(a) \, = \, f(a) , 
\]
we get
\[
(I^{(M)}_{a+,g} \, D^{(K),*}_{a+,g} \, f)(x) \, = \, 
Q_g \, \Bigl( (Q^{-1}_g \, f)(x) \, - \, 
(Q^{-1}_g \, f)(a) \Bigr) \, = \, 
f(x) \, - \, f(a) ,
\]
which was to be proved. 

For the right-sided parametric GFD, the identity is proved similarly.
Q.E.D.
\end{proof}

Note that only such differential and integro-differential operators that satisfy generalized analogues of fundamental theorems can be called generalized integrals and derivatives.
In this case, these operators form a fractional calcalus. 
The parametric GFIs and GFDs satisfy 
the first fundamental Theorems \ref{FT-1-RL} and \ref{FT-1-C},
the second fundamental Theorems
\ref{FT-2-RL} and \ref{FT-2-C}.
Therefore The parametric GFIs and GFDs satisfy 
for a general fractional calculus, 
which can be called parametric GFC.

%%%%%%%%%%%%%%%%%%%%%%%%%%%%%%%%%%%%%%%%%%%%%%%%%%%%%%%%%%%%%
%%%%%%%%%%%%%%%%%%%%%%%%%%%%%%%%%%%%%%%%%%%%%%%%%%%%%%%%%%%%%
%%%%%%%%%%%%%%%%%%%%%%%%%%%%%%%%%%%%%%%%%%%%%%%%%%%%%%%%%%%%%

%%%\newpage

\section{Economic Interpretation of Parametric GF Derivatives}

Let us briefly explain a possible applications and interpretation of the parametric GF derivatives as extensions of the standard economic marginal value and marginal values with memory.

In the study of economic processes, as a rule, the marginal values are calculated for the various economic indicators, which are presented as functions of some factors 
\cite{Varian,Varian2010,Gossen-G,Gossen,BOOK-DG-2021}. 
The standard concept of marginal value includes but is not limited to the following notions: 
the marginal utility (MU); 
the marginal product (MP); 
the marginal cost (MC); 
the marginal benefit (MB);
the marginal rate of substitution (MRS);
the marginal revenue product (MRP); 
the marginal product of capital (MPK); 
the marginal product of labor (MPL); 
the marginal rate of transformation (MRT); 
the marginal propensity to consume (MPC); 
the marginal propensity to save (MPS); 
the marginal tax rate (MTR). 
The concept of marginal value allows the use of standard mathematical calculus of derivatives of integer orders to describe changes in economic processes. Usually, the marginal value is defined as the first-order derivative of the function of a certain index with respect to its determining factor. The standard marginal value shows the increase in the corresponding indicator per unit increase in its determining factor. 

To define the standard marginal value for the economic indicator $Y$, one can consider a function $Y=Y(X)$ 
describing the dependence of an indicator $Y$ on a factor
$X$.
If the given function $Y=Y(X)$ is single-valued and differentiable, 
then the marginal value (${MY}_X$) of the indicator $Y$ is defined as the first derivative of the function $Y=Y(X)$ 
with respect to $X$ as
\begin{equation} \label{EQ12-2} 
{MY}_X = \frac{dY(X)}{dX} . 
\end{equation} 

In many cases, the indicator and the factor can be considered as single-valued functions of time $t$. 
Therefore, to define the marginal values of the indicators, one can use parametric dependence of an indicator $Y$ on a factor $X$ in the form of the single-valued functions $Y=Y(t)$ and $X=X(t)$ with the parameter $t$.
Therefore, the marginal values of indicators can defined by the following form.

\begin{Definition}
Let $Y=Y(t)$ and $X=X(t)$ are single-valued differentiable functions with respect to a variable $t$, which describe the parametric dependence of the economic indicator $Y$ on some factor $X$. 
Then, the standard marginal values at time $t$ 
are defined by the equations 
\begin{equation} \label{EQ12-6} 
{MY}_X (t) \, = \, \frac{dY(t)/dt}{dX(t)/dt}, 
\end{equation} 
where $dX(t)/dt \ne 0$. 
\end{Definition}

If the dependence of $Y(t)$ on $X(t)$ can be represented as a single-valued differentiable function $Y=Y(X)$, by excluding the time parameter $t$, then equations \eqref{EQ12-2} and \eqref{EQ12-6} are equivalent by the standard chain rule
\begin{equation} \label{New-1}
\frac{dY(X(t))}{dt} \, = \, 
\left(\frac{dY(X)}{dX}\right)_{X=X(t)} \frac{dX(t)}{dt} .
\end{equation} 
Therefore if the function $X=X(t)$ is invertible in a neighborhood of the point $t$, then
\begin{equation} \label{New-2}
{MY}_X(t) \, = \, \frac{dY(t)/dt}{dX(t)/dt}=
\frac{dY(X)}{dX} \, = \, {MY}_X . 
\end{equation} 
It should be noted that standard chain rule is violated for fractional derivatives and GFDs 
\cite{CNSNS2016,Cresson2020,Math2019}.

In order for equations \eqref{New-1} and \eqref{New-2} to be satisfied, it is sufficient that the function $X = X(t)$ be reversible.
Equation \eqref{EQ12-6} gives a standard definition of the parametric derivative of the first order, 
if the function $X=X(t)$ has an inverse function in a neighborhood of $t$ and the functions $X=X(t)$ and $Y=Y(t)$ have the first derivatives. 
Equation \eqref{EQ12-6} can be considered as a parametric derivative of the indicator $Y = Y(t)$ by the factor $X = X(t)$ at time $t$, if $dX (t)/dt \ne 0$. 

Memory leads to the fact that the marginal values of an economic indicator at the time $t$ can depend on history of changes of the variables $Y(t)$ and $X(t)$ on a finite time interval. 
The standard marginal values of indicator depend only on the given time $t$ and its infinitesimal neighborhood. 
Equation \eqref{EQ12-6} cannot be used for processes with memory, because the derivatives of integer order 
are determined by the behavior of the functions 
$X =X(t)$ and $Y=Y(t)$ in an infinitely small neighborhood of the time instant $t$, which means instant forgetting (amnesia) of all the changes that were made before.
Therefore, the standard definitions of these values are applicable only on the condition that all economic agents have full amnesia. Because of this, it is necessary to generalize the concept of marginal value to take into account the effects of memory in economic processes. 
To describe the processes with memory, mathematical tools should allow us to take into account history of changes of variables $Y(t)$ and $X(t)$ on a finite time interval. 
Such a mathematical tool could be fractional calculus.

%%%%%%%%%%%%%%%%%

In 2016, new concepts that generalize the standard marginal value were proposed in \cite{Margin1,Margin2,Margin3} 
(see also works \cite{Margin4,Margin5,Interpret,BOOK-DG-2021}). 
The use of fractional calculus allows us to apply an economic analysis to describe economic processes with memory. 
The fractional calculus generalizations of standard marginal values were suggested in \cite{Margin1,Margin2,Margin3,Margin4,Margin5,BOOK-DG-2021}. 
These generalizations take into account the memory, which is described as a dependence of economic processes on changes 
of indicators and factors in the past on a finite time interval. 

%%%%%%%%%%%%%%%%%

The generalizations of the standard marginal value 
can be defined by using parametric fractional derivatives of non-integer order, which are also called the fractional derivatives of function with respect to another function. 
The Caputo parametric fractional derivative was proposed in Definition 3 of paper \cite[p.~224]{Elasticity1}, and then the properties of this derivative were described in \cite{Almeida-1}. 

Let $X(\tau)$ be an increasing positive monotonic function having a continuous derivative $X^{(1)}(t)$. 
Then the Caputo parametric fractional derivative of the order $\alpha \ge 0$ is defined by the equation
\begin{dmath} \label{EQ12C-1}
\left(D^{\alpha}_{X(t)} Y\right)(t) = 
\frac{1}{\Gamma (n -\alpha)} \int^t_0 d\tau 
\frac{X^{(1)}(\tau)}{(X(t)-X(\tau))^{\alpha +1-n}} 
\left( \frac{1}{X^{(1)}(\tau)}
\frac{d}{d\tau}\right)^n Y (\tau) , 
\end{dmath} 
where $n - 1 < \alpha \le n$ and $0<\tau <t$.

If $X(t)=t$, then equation \eqref{EQ12C-1} gives the well-known Caputo fractional derivative 
\begin{equation}
\left(D^{\alpha}_{X(t)} Y \right)(t) =
\left(D^{\alpha}_{C;0+} Y\right) (t) .
\end{equation} 
For the parametric Caputo derivative \eqref{EQ12C-1} of the function $Y(t)=(X(t)-X(0))^{\beta}$, we have the equation
\begin{equation} \label{EQ12C-2} 
\left(D^{\alpha}_{X(t)} Y\right)(t) = 
\frac{\Gamma (\beta +1)}{\Gamma (\beta -\alpha +1)}
(X(t)-X(0))^{\beta -\alpha} , 
\end{equation} 
where $\beta -\alpha >0$.

The generalized marginal values was defined by the using the parametric Caputo fractional derivative \cite{BOOK-DG-2021}.
If $Y(t)$ can be represented as a single-valued function of $X(t)$, then we can define the marginal value with memory
by using the parametric Caputo fractional derivatives
\eqref{EQ12C-1}, which is the fractional derivatives of the function $Y(t)$ with respect to the function $X(t)$. 

\begin{Definition}
Let $Y(t)$ be a function that can be represented as a single-valued function of $X(t)$, which is an increasing positive monotonic function having a continuous derivative $X^{(1)}(t) \ne 0$. Then, the marginal value of the order $\alpha \ge 0$ can be defined by the equation
\begin{equation} \label{EQ12C-12}
MY_X(\alpha ,t) \, = \, 
\left(D^{\alpha}_{X(t)} Y\right)(t) , 
\end{equation} 
where $D^{\alpha}_{X(t)}$ is the Caputo parametric derivative \eqref{EQ12C-1}.
\end{Definition}

In the case $Y(t) = X(t)$ and $Y(t)=X(t)-X(0)$, using equation \eqref{EQ12C-2} for expressions \eqref{EQ12C-12} and $\left(D^{\alpha}_{X(t)} X(0)\right)(t)=0$, equation \eqref{EQ12C-12} gives
\begin{equation} \label{EQ12C-15-1} 
MX_X(\alpha ,t) = \frac{1}{\Gamma (2-\alpha)}
(X(t)-X(0))^{1-\alpha} . 
\end{equation} 
In order to the marginal value of non-integer order gives one at $Y(t)=X(t)$, i.e., $MX_X(\alpha,t)=1$, we can define the marginal value of fractional order by the equation
\begin{equation} \label{EQ12C-15-2} 
MY^{*}_X(\alpha ,t) \, = \, 
\frac{\left(D^{\alpha}_{X(t)} Y\right)(t)}{\left(D^{\alpha}_{X(t)} X \right)(t)} 
\end{equation} 
instead of equation \eqref{EQ12C-12}.

Using the GF derivative, one can defined the GF marginal value by the equations
\begin{equation} \label{EQ12C-15GF1} 
MY_X(K,t) =
\left(D^{(K)}_{a+,X(t)} Y\right)(t) ,
\end{equation} 
\begin{equation} \label{EQ12C-15GF2} 
MY^{*}_X(K,t) =
\frac{\left(D^{(K)}_{a+,X(t)} Y\right)(t)}{\left(D^{(K)}_{a+,X(t)} X \right)(t)} ,
\end{equation} 
where $t \, > \, a$.

For the kernel $K(t) = h_{1-\alpha}(t)$, equations
\eqref{EQ12C-15GF1} and \eqref{EQ12C-15GF2} give
\eqref{EQ12C-15-1} and \eqref{EQ12C-15-2}, i.e.
\begin{equation} 
MY_X(h_{1-\alpha},t) \, = \, MY_X(\alpha,t) ,
\qquad
MY^{*}_X(h_{1-\alpha},t) \, = \, MY^{*}_X(\alpha,t) .
\end{equation}

Therefore equations \eqref{EQ12C-15GF1} and \eqref{EQ12C-15GF2} give a possible interpretations of the parametric GFDs in economics \cite{BOOK-DG-2021}.
The economic interpretation of the parametric GFDs is based on a generalization of marginal values of economic indicators, which take into account non-locality in time and memory.
The parametric GFDs are interpreted as economic characteristics that are intermediate between the standard average and marginal values of indicators 
\cite{TT-1,TT-2,Rehman2018,BOOK-DG-2021}.

The economic interpretation of the parametric GFDs can be also based on a general fractional extension of the concept of fractional elasticity, proposed \cite{Elasticity1} to take into account memory in economic dynamics. 
The GF elasticity can be defined as 
\begin{equation} \label{EQ12C-15GF-E} 
E_{{(K)}}(Y(t);X(t)) =
\left(D^{(K)}_{a+, \ln X(t)} \ln Y \right)(t) ,
\end{equation} 
that can be considered of GF extension of the fractional Log-elasticity proposed in Definition 3 of \cite{Elasticity1} and 
Section 5 of \cite{BOOK-DG-2021}, pp.101-121.
Equation \eqref{EQ12C-15GF-E} defines an elasticity with memory that is described by operators with kernels of Sonin type insted of FDs with power-law kernels 
\cite{Elasticity1,BOOK-DG-2021}.

%%%%%%%%%%%%%%%%%%%%%%%%%%%%%%%%%%%%%%%%%%%%%%%%%%%%%%%%%%%%%
%%%%%%%%%%%%%%%%%%%%%%%%%%%%%%%%%%%%%%%%%%%%%%%%%%%%%%%%%%%%%
%%%%%%%%%%%%%%%%%%%%%%%%%%%%%%%%%%%%%%%%%%%%%%%%%%%%%%%%%%%%%

\section*{Statements and Declarations}

\ \ \ \
Funding: This research received no external funding. \\

Institutional Review Board Statement: Not applicable.\\

Informed Consent Statement: Not applicable.\\

Data Availability Statement: Not applicable.\\

Conflicts of Interest: The author declare no conflict of interest.

%%%%%%%%%%%%%%%%%%%%%%%%%%%%%%%%%%%%%%%%%%%%%%%%%%%%%%%%%%%%%
%%%%%%%%%%%%%%%%%%%%%%%%%%%%%%%%%%%%%%%%%%%%%%%%%%%%%%%%%%%%%
%%%%%%%%%%%%%%%%%%%%%%%%%%%%%%%%%%%%%%%%%%%%%%%%%%%%%%%%%%%%%

%%%\newpage

%%%%%%%%%%%%%%%%%%%%%%%%%%%%%%%%%%%%%%%%%%%%%%%%%%%%%%%%%%%%%%
%%%%%%%%%%%%%%%%%%%%%%%%%%%%%%%%%%%%%%%%%%%%%%%%%%%%%%%%%%%%%%
%%%%%%%%%%%%%%%%%%%%%%%%%%%%%%%%%%%%%%%%%%%%%%%%%%%%%%%%%%%%%%

\end{document}